\documentclass{amsart}
\usepackage{fullpage}
\usepackage[foot]{amsaddr}
\usepackage{tikz}
\usepackage{hyperref}
\usepackage{amsmath}
\usepackage{pgfmath}
\usepackage{amsthm}
\usepackage{amsfonts}
\usepackage{amssymb}
\usepackage{graphicx}
\usepackage{float}
\usepackage{ytableau}
\usepackage[all]{xy}
\usepackage[normalem]{ulem}
\usepackage{enumitem}

\setlist[itemize]{align=parleft,left=0pt..1em}

\newcommand{\stkout}[1]{\ifmmode\text{\sout{\ensuremath{#1}}}\else\sout{#1}\fi}

\newtheorem{theorem}{Theorem}
\newtheorem*{theorem*}{Theorem}

\newtheorem{proposition}[theorem]{Proposition}

\newtheorem{corollary}[theorem]{Corollary}

\newtheorem{lemma}[theorem]{Lemma}

\newcommand{\Ad}{\mathrm{Ad}}

\newcommand{\Ind}{\mathrm{Ind}}

\newcommand{\bbA}{\mathbb{A}}

\newcommand{\frakg}{\mathfrak{g}}

\newcommand{\frakl}{\mathfrak{l}}

\newcommand{\fraku}{\mathfrak{u}}
\newcommand{\frakn}{\mathfrak{n}}

\newcommand{\diag}{\mathrm{diag}}

\newcommand{\calA}{\mathcal{A}}

\title{An Analogue of Bernstein-Zelevinsky Derivatives to Automorphic Forms}
\author{Zhuohui Zhang\\Tel Aviv University}
\email{zhuohui.zhang@weizmann.ac.il} 
\begin{document}

\begin{abstract}
    In this paper, a construction to imitate the Bernstein-Zelevinsky derivative for automorphic representations on $GL_n(\mathbb{A})$ is introduced. We will later consider the induced representation \[I(\tau_1,\tau_2;\underline{s}) = \Ind_{P_{[n_1,n_2]}}^{G_n}(\Delta(\tau_1,n_1)|\cdot|^{s_1}\boxtimes \Delta(\tau_2,n_2)|\cdot|^{s_2}).\] from the discrete spectrum representations of $GL_n(\mathbb{A})$, and apply our method to study the degenerate Whittaker coefficients of the Eisenstein series constructed from such a representation as well as of its residues. This method can be used to reprove the results on Whittaker supports of automorphic forms of such kind proven by D. Ginzburg, Y. Cai and B. Liu. This method will also yield new results on the Eulerianity of certain degenerate Whittaker coefficients.
\end{abstract}

\maketitle

\section{Introduction}
The Bernstein-Zelevinsky derivative is a tool introduced by Bernstein and Zelevinsky in \cite{bernstein1977induced} to study the representations of general linear groups over a $p$-adic field. A classification of irreducible representations of $GL_n$ modulo supercuspidal representations was obtained with this tool in \cite{zelevinsky1980induced}. Following the notations introduced in \cite{bernstein1977induced}, denoting by $G_n = GL_n$ and $P_n = G_{n-1}\rtimes k^{n-1}$ the \emph{mirabolic} subgroup in $G_n$, defined as the collection of elements in $G_n$ with last row $(0,0,\ldots,0,1)$. The \emph{Bernstein-Zelevinsky derivative} is based on a collection of induction-restriction functors:
\begin{align*}
    \xymatrix{
        \mathrm{Rep}P_n \ar@<2pt>[rr]^{\Psi^- = r_{U,1}} &&\mathrm{Rep} G_{n-1}\ar@<2pt>[ll]^{\Psi^+=i_{V,1}}
    }\\
    \xymatrix{
        \mathrm{Rep}P_n \ar@<2pt>[rr]^{\Phi^-= r_{U,\theta}} && \mathrm{Rep} P_{n-1}\ar@<2pt>[ll]^{\Phi^+=i_{V,\theta}, \hat{\Phi}^+=i^c_{V,1}}
    }
\end{align*}
and is defined as the composition of functors:
\[
    (\cdot)^{(k)}  = \Psi^-\circ (\Phi^-)^{k-1}: \xymatrix{
        \mathrm{Rep}G_n \ar[rr] &&\mathrm{Rep} G_{n-k}.
    }
\]
These functors satisfy a collection of adjunction properties as described in \cite[3.2]{bernstein1977induced} that are ultimately related to the properties of the Whittaker models of representations over a $p$-adic group. Calculation with these functors is based on the \emph{geometric lemma}, which is a combinatorial rule for calculating the composition of a restriction functor applied after an induction functor. For two standard Levi subgroups $M_{\alpha}, M_{\beta}$ whose sizes of Levi blocks are prescribed by two partitions $\alpha,\beta$, it was shown in \cite[2.12]{bernstein1977induced} that the functor $F = r_{G,M_{\beta}}\circ i_{G, M_{\alpha}}$ is \emph{glued} from the functors $F_w = i_{M_{\beta}, M_{\beta'}}\circ w \circ r_{M_{\alpha'}, M_{\alpha}}$, where
\begin{itemize}
    \item The partitions $\alpha = (\alpha_1,\ldots,\alpha_r), \beta = (\beta_1,\ldots,\beta_s)$ prescribe the block sizes of the standard Levi subgroups $M_{\alpha}, M_{\beta}$.
    \item The Weyl group element $w$ are those that sends any positive root $\gamma$ of $M_\alpha$ to a positive root $w\alpha$, and $w^{-1}$ sending a positive root $\epsilon$ of $M_\beta$ to $w^{-1}\epsilon$.
    \item The refined partition $\gamma'$ corresponds to the Levi subgroup $M_\gamma\cap w M_\beta w^{-1}$, while $\beta'$ corresponds to the Levi subgroup $M_\beta\cap w^{-1} M_\gamma w$.
\end{itemize}
In this paper, we will replace the functorial operations for $p$-adic group representations by Eisenstein series and their Whittaker coefficients, and imitate the method in \cite{bernstein1977induced} and \cite{zelevinsky1980induced} to prove results about the Whittaker support of automorphic forms and the Eulerianity of Whittaker coefficients of automorphic forms for $GL_n$.\\

By the work of Langlands summarized in \cite[Section V.3]{moeglin1995spectral}, a square-integrable automorphic form can be described as residues of Eisenstein series induced from discrete spectrum data on the Levi subgroups of parabolic subgroups. In particular, it is shown by M\oe glin-Waldspurger in \cite{moeglin1989spectre} that the discrete spectrum of $GL_n$ consists of the generalized Speh representations $\Delta(\tau,n)$, which is the irreducible representation generated by the residue at $\underline{s} = (s_1,\ldots,s_n) \rightarrow (\frac{n-1}{2},\ldots,-\frac{n-1}{2})$ of an Eisenstein series
\[
    E(\cdot,\underline{s}): \Ind_{P_{[a^n]}}^{G_{an}}\left(\tau|\cdot|^{s_1}\boxtimes\ldots\boxtimes \tau|\cdot|^{s_n}\right)\rightarrow \mathcal{A}(G_n(k)\backslash G_n(\bbA)).
\]
Using the analogues of Bernstein-Zelevinsky derivative developed in this paper, we can mimic some arguments in \cite{zelevinsky1980induced} for induced representations attached to discrete spectra data. This method can be applied to prove results regarding nilpotent orbit associated to the top Fourier coefficient. Some of them have already been provided in \cite{ginzburg2006certain}, \cite{liu2013fourier}, \cite{cai2018fourier} and \cite{liu2020top}.

\section{Eisenstein Series and Whittaker Coefficients}
\subsection{Eisenstein Series and Degenerate Whittaker Models}
For any standard parabolic subgroup $P=MN\subset G_n$ and an automorphic representation $\pi = \bigotimes_v\pi_v$ of $M$, as an irreducible subrepresentation of the space $\calA\left(M(k)\backslash M(\bbA)\right)$ of all automorphic forms on $M$. For the group $G_n = GL_n$ and a partition $\underline{\alpha} = (n_1,n_2,\ldots,n_r)$ of the integer $n$, denote by $P_{\underline{\alpha}}$ the standard parabolic subgroup with a Levi subgroup $M_{\underline{\alpha}} \cong GL_{n_1}\times GL_{n_2}\times\ldots\times GL_{n_r}$ embedded diagonally in $G_n$. The representation $\pi_{\underline{s}}$ on $M_n$ is defined on a pure tensor $v_1\otimes\ldots\otimes v_r$ as
\[
    \pi_{\underline{s}}(\mathrm{diag}\left(m_1,\ldots,m_r\right))\left(v_1\otimes\ldots\otimes v_r\right) = \pi_1^{s_1}(m_1)v_1\otimes\ldots\otimes\pi_r^{s_r}(m_r)v_r
\]
where $\pi_i^{s_{i}} = \pi_i \left|\cdot\right|^{s_i}$ for $\mathrm{Re}s_i>0$. The induced representation $I(\pi_{\underline{s}})$ is the space of $K$-finite functions $\varphi_{\underline{s}}:G_n\rightarrow \pi$ satisfying
\[
    \varphi_{\underline{s}}(nmg) = \pi_{\underline{s}}(m)\varphi_{\underline{s}}(g).
\]
One can define $\tilde{\varphi}_{\underline{s}}$ on $G_n$ as $\varphi(g)$ evaluated at the identity element of $M_{\underline{n}}$. The Eisenstein series $E_{\underline{s}}$ is thus an operator from $I(\pi_{\underline{s}})$ to the space of automorphic forms $\mathcal{A}(G_n(k)\backslash G_n(\bbA))$ defined as
\[
    E(\varphi_{\underline{s}},g) = \sum_{\gamma\in P(k)\backslash G(k)} \tilde{\varphi}_{\underline{s}}(\gamma g).
\]
Picking a different standard parabolic subgroup $Q_\beta = L_\beta U_\beta$ corresponding to a different partition $\beta$, and an unramified additive character $\psi$ on $U$, the \emph{Whittaker coefficient} of any automorphic form $f\in\mathcal{A}\left(G(k)\backslash G(\bbA)\right)$ can be defined as
\[
    \mathcal{W}_\psi (f)(g) = \int_{U_\beta(k)\backslash U_\beta(\bbA)} f(ug)\overline{\psi(u)}du.
\]
Denoting by $L_\psi$ the stabilizer of $\psi$ in $L$, the function $\mathcal{W}_\psi (f)$ lives in the space $\mathcal{A}(L_\psi(k)U(\bbA)\backslash G(\bbA))_\psi$, which is defined as the $L_\psi(k)$-invariant functions $w$ satisfying the property $w(ug)=\psi(u)w(g)$ for each $u\in U(\bbA)$.\\

The $\emph{Bernstein-Zelevinsky}$ derivatives are a collection of operators first introduced in \cite{bernstein1977induced} for representations of $GL_n$ over a non-archimedian local field. Following our notations introduced above, let $P_{[n_1,n_2]}$ be a standard parabolic subgroup with Levi blocks $M_{[n_1,n_2]}=GL_{n_1}\times GL_{n_2}$, and $Q_{[n-1,1]}$ a standard parabolic subgroup with Levi blocks $L_{[n-1,1]} = GL_{n-1}\times GL_1$. Denoting the character
\[
    \psi_{m_1,\ldots,m_{n-1}}\begin{pmatrix}
         1&\ldots&0 &u_{n-1}\\
          &\ldots &0 & \cdots\\
         &  & 1 & u_{1}\\
         & & & 1
    \end{pmatrix} = \psi\left(\sum_i m_i u_i\right).
\]
on the corresponding unipotent radical $U_{[n-1,1]}\cong \mathbb{A}^{n-1}$ by $\psi_{m_1,\ldots,m_{n-1}}$, similar to the definition of the corresponding operators in \cite{bernstein1977induced}, the following two operators $\Psi$ and $\Psi$ are defined as special cases of the Whittaker operator $\mathcal{W}_\psi$:
\begin{align}
    \Psi :\mathcal{A}\left(G_n\left(k\right)\backslash G_n\left(\mathbb{A}\right)\right)&\longrightarrow\mathcal{A}\left(L_{[n-1,1]}(k)U_{[n-1,1]}(\bbA)\backslash G_n\left(\mathbb{A}\right)\right)\label{psiop}\\
    f&\longmapsto \mathcal{W}_{\psi_{0\ldots 0}} (f)\nonumber
\end{align}
and
\begin{align}
    \Phi :\mathcal{A}\left(G_n\left(k\right)\backslash G_n\left(\mathbb{A}\right)\right)&\longrightarrow\mathcal{A}\left(Q'(k)U_{[n-1,1]}(\bbA)\backslash G_n\left(\mathbb{A}\right)\right)_{\psi_{10\ldots 0}}\label{phiop}\\
    f&\longmapsto \mathcal{W}_{\psi_{10\ldots 0}} (f)\nonumber
\end{align}
where $Q'$ is the stabilizer $GL_{n-2}\ltimes V_{n-2}$, a \emph{mirabolic} subgroup of $GL_{n-1}$. The $r$-th Bernstein-Zelevinsky derivative operator 
\[
    D^{(m)}:\calA\left(G_n(k)\backslash G_n(\bbA)\right)\longrightarrow \calA\left(G_{n-m}(k)\times Z(G_m)U'_{m}(\bbA)\backslash G(\bbA)\right)_{\psi_{m}}
\]
($r$ is the length of the partition)
is defined as the composition $\Psi\circ (\Phi)^{m-1}$, where $G_{n-m}$ is the subgroup in $GL_{m}$ embedded on the top-left corner of $GL_n$, $U'_m$ the nilpotent radical of a standard parabolic subgroup of corresponding to the partition $[n-m,1^{m}]$, and $\psi_m$ the character
\[
    \psi_m\left(\begin{smallmatrix}
        1 & \ldots& 0 & u_{k,n-k} & & \ldots &u_{1,n-1}\\
        & 1 & \ldots & & & &\\
        & & 1 & u_{k,1} & & &\\
        & & & 1 & & \ldots & \ldots\\
        & & & & & 1 & u_{1,1}\\
        & & & & & & 1
    \end{smallmatrix}\right) = \psi\left(0\cdot u_{k,1}+ u_{k-1,1} + \ldots + u_{1,1}\right)
\]
on $U'_m$. One can also write the derivative operator just as a Whittaker coefficient $\mathcal{W}_{\psi_m}(f)$. For any automorphic form $f$ of $G$, the Bernstein-Zelevinsky derivative $D^{(m)}(f)$ can be restricted to an automorphic function on $G_{n-m}$.

\subsection{Constant Term Formula}
The Bernstein-Zelevinsky derivative operator $D^{(m)}$ can also be viewed as the composition of the constant term operator $c_{P_{[n-m,m]}}$ on $P_{[n-m,m]}$ and a generic Whittaker coefficient operator $\mathcal{W}_{m}$ on the second factor $GL_{m}$:
\[
    \xymatrix{
         & \calA (L_{[n-m,m]}(k)U_{[n-m,m]}(\bbA)\backslash G_n(\bbA)) \ar[dd]^{\mathcal{W}_m}\\
         \calA (G_n(k)\backslash G_n(\bbA))
         \ar[ur]^{c_{P_{[n-m,m]}}} \ar[rd]^{D^{(m)}}& \\
         & \calA\left(G_{n-m}(k)\times Z(G_m)U'_{m}(\bbA)\backslash G(\bbA)\right)_{\psi_{m}}
    }
\]
Following a similar method as in \cite[II.1.7]{moeglin1995spectral} the constant term along any unipotent radical $U'$ of the parabolic subgroup $P' = M'U'$ of the Eisenstein series $E(\tilde{\varphi},g)$ can be written as the integral:
\[
    c_{P'}E(\tilde{\varphi},g)=\sum_{\gamma\in P(k)\backslash G(k)}\int_{U'(k)\backslash U'(\mathbb{A})}\varphi(\gamma u' g)du'.
\]
By the Bruhat decomposition $G(k) = \bigcup_{w\in W_{M,M'}}P(k)wP'(k)$ with $W_{M,M'} = P\backslash G/P'$, any element $\gamma$ belongs to a coset $Pwp'$ for some $p'\in P'$. This $p'$ can be further decomposed into $p' = m'u'$:
\begin{align*}
    c_{P'}E(\tilde{\varphi},g)&=\sum_{w\in W_{M,M'}}\sum_{p'\in P'(k)}\int_{U'(k)\backslash U'(\mathbb{A})}\varphi(w p' u' g)du'\\
    &=\sum_{w\in W_{M,M'}}\sum_{m'\in M'(k)\cap w^{-1}P'(k)w\backslash M'(k)}\int_{U'(\mathbb{A})}\varphi(w m' u' g)du'\\
    &=\sum_{w\in W_{M,M'}}\sum_{m'\in M'(k)\cap w^{-1}P'(k)w\backslash M'(k)}\int_{U'(\mathbb{A})}\varphi(w u' m' g)du'\\
    &=\sum_{w\in W_{M,M'}}\sum_{m'\in M'(k)\cap w^{-1}P'(k)w\backslash M'(k)}\int_{U'(\mathbb{A})\cap w^{-1}U(\mathbb{A})w\backslash U'(\mathbb{A})}\varphi(w u' m' g)du'.
\end{align*}
We can decompose the domain of integration as the product of the following spaces:
\[
    U'(\mathbb{A})\cap w^{-1}U(\mathbb{A})w\backslash U'(\mathbb{A}) = \left(U'(\mathbb{A})\cap w^{-1}M(\mathbb{A})w\right) \left(U'(\mathbb{A})\cap w^{-1}\overline{U}(\mathbb{A})w\right).
\]
In our setting, we take $P = P_{[n_1,n_2]}$ and $P'=Q_{[n-m,m]}$. Since the unipotent group $U_{[n-m,m]}$ is abelian, the subgroup $U'(\mathbb{A})\cap w^{-1}M(\mathbb{A})w$ is isomorphic to the unipotent subgroup $U_{n_1}\times U_{n_2}$. We can thus decompose the constant term $c_{P'}E(\tilde{\varphi},g)$ as the summation over the Weyl group double coset $W_{M,M'}$ and $M'\cap w^{-1}Pw\backslash M'$:
\[
    c_{P'}E(\tilde{\varphi},g)
    =\sum_{w\in W_{M,M'}}\sum_{m'\in M'(k)\cap w^{-1}P'(k)w\backslash M'(k)}\mathcal{A}_{w}\left(c_{P'\cap w^{-1}Mw}\varphi\right)(wm'g),
\]
where the operator $\mathcal{A}_w$
\begin{equation}\label{intertwiningop}
    \mathcal{A}_w f(g) = \int_{U'(\mathbb{A})\cap w^{-1}\overline{U}(\mathbb{A})w}f(u'g)du'.
\end{equation}
is a formal intertwining operator between the following two induced representations:
\[
    \mathcal{A}_w: \mathrm{Ind}_{wM'w^{-1}\cap M}^{G}(\pi)\rightarrow \mathrm{Ind}_{M'\cap w^{-1}Mw}^{G}(\pi^w)
\]
for any representation $\pi$ on the intersection of Levi subgroups $wM'w^{-1}\cap M$. If the induction data is associated to the representation $\tau^s$, the operator $\mathcal{A}_w$ can be analytically continuated to the whole $\mathbb{C}^n$ as a meromorphic function of $s$.\\

We denote the generic Whittaker operator on automorphic functions on the lower-right corner $GL_m$ of the Levi subgroup by $\mathcal{W}_m$, which sends any function $f\in\calA (M_{[n-m,m]}(k)U_{[n-m,m]}(\bbA)\backslash G(\bbA))$ to a function in the space $\calA (M_{[n-m,1^m]}(k)U_{[n-m,1^m]}(\bbA)\backslash G(\bbA))_{\psi_s}$. The operator can be written as the integral
\[
    \left(\mathcal{W}_m f\right)(g) =\int_{U_m(k)\backslash U_m(\bbA)} f\left(\left(\begin{smallmatrix}
        I_{n-m} & 0 & 0\\
        0 & 1 & u\\
        0 & 0 & 1
    \end{smallmatrix}\right)g\right)\overline{\psi_{[m]}(u)}du
\]
where $\psi_{[m]}$ is character on $U_m$ belonging to the regular orbit of the space $GL_{m}$. Now we can view the derivative operator as the sum of compositions of three operators:
\begin{align}\label{geomlemma}
    D^{(m)} E(\tilde{\varphi},g)
    =\sum_{w\in W_{M,M'}}\mathcal{W}_m\left(\sum_{m'\in M'(k)\cap w^{-1}P'(k)w\backslash M'(k)}\mathcal{A}_{w}\left(c_{P'\cap w^{-1}Mw}\varphi\right)(wm'g)\right)
\end{align}
with 
\begin{itemize}
    \item $P = P_{[n_1,n_2]} = M_{[n_1,n_2]}U_{[n_1,n_2]}$, with the corresponding $M = M_{[n_1,n_2]}$ and $U = U_{[n_1,n_2]}$,
    \item $P' = P_{[n-m,m]} = M_{[n-m,m]}U_{[n-m,m]}$, with the corresponding $M' = M_{[n-m,m]}$ and $U' = U_{[n-m,m]}$,
    \item $W_{M,M'}$ is the double coset space $P\backslash G/P'$, which is the set of all $w\in W(G)$ such that any positive root $\alpha$ of $M$ is sent to a positive root $w\alpha$, and any positive root $\beta$ of $M'$ is sent to a positive root $w^{-1}\beta$.
    \item The intertwining operator $\mathcal{A}_w$ is defined by the analytic continuation of the integral (\ref{intertwiningop}).
\end{itemize}

\subsection{Weyl Group Cosets}
In this section, we will give a combinatorial description to the representatives of the double coset $P_{[n_1,n_2]}\backslash G/Q_{[n-m,1^m]}$, as well as the set of Weyl group representatives in $W_{M,M'}$ such that the corresponding term in the sum of (\ref{geomlemma}) is nonzero. For each pair of integers $l\leq k$, we denote the \emph{cycle} $(k+1,k,\ldots,l+1,l)$ in the permutation group by $c_{k,l}$. We need two technical lemma:
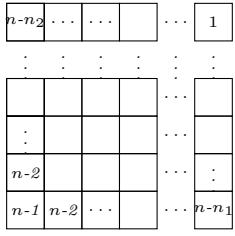
\begin{lemma}\label{lemmabij1}
    There is a bijection between the set of interlacings between two strings $12\ldots n_1$ and $(n_1+1)\ldots (n_1+n_2)$ and the set of Young subdiagrams $[n_1-k_1,n_1-k_2,\ldots,n_1-k_{n_2}]$ with $0\leq k_1\leq k_2\ldots \leq k_{n_2}\leq n_1$ of $[n_1^{n_2}]$.
\end{lemma}
\begin{proof}
    In the permutation group $S_{n_1+n_2}$, the cycle $c_{l,k}$ is the product of simple reflections:
    \[
        c_{l,k} = s_l s_{l+1}\ldots s_{k}.
    \]
    We can correspond any Young subdiagram $[n_1-k_1,n_1-k_2,\ldots,n_1-k_{n_2}]$ to the element
    \[
        w_{[n_1-k_1,n_1-k_2,\ldots,n_1-k_{n_2}]} =  c_{n-n_1+k_{n_2},n-1}c_{n-n_1-1+k_{n_2-1},n-2}\ldots c_{1+k_1,n-n_2}
    \]
    which sends the word $12\ldots (n_1+n_2)$ to a word $a_1a_2\ldots a_{n_1+n_2}$ with
    \[
        \left(a_{k_1+1},a_{k_2+2},\ldots,a_{k_{n_2}+n_1}\right) = \left(n_1+1,n_1+2,\ldots,n_1+n_2\right)
    \]
    and the complement of $\left(a_{k_1+1},a_{k_2+2},\ldots,a_{k_{n_2}+n_1}\right)$ equal to $12\ldots n_1$. Thus, we have constructed a bijection between the set of interlacings between the two strings and the set of Young subdiagrams of $[n_1^{n_2}]$.
\end{proof}

\begin{lemma}\label{lemmainverse}
    For the rectangular Young diagram $[n_1^{n_2}]$:
    \begin{center}
        \begin{tikzpicture}
            \draw (-0.25,-0.25) rectangle (0.25,0.25);
            \draw (0.25,-0.25) rectangle (0.75,0.25);
            \draw (0.75,-0.25) rectangle (1.25,0.25);
            \draw (1.25,-0.25) rectangle (1.75,0.25);
            \draw (2,0) node {$\ldots$} ;
            \draw (2.25,-0.25) rectangle (2.75,0.25);

            \draw (-0.25,0.25) rectangle (0.25,0.75);
            \draw (-0.25,0.75) rectangle (0.25,1.25);
            \draw (-0.25,1.25) rectangle (0.25,1.75);
            \draw (0,2) node {$\vdots$};
            \draw (0.5,2) node {$\vdots$};
            \draw (1,2) node {$\vdots$};
            \draw (1.5,2) node {$\vdots$};
            \draw (2,2) node {$\vdots$};
            \draw (2.5,2) node {$\vdots$};
            \draw (-0.25,2.25) rectangle (0.25,2.75);

            \draw (-0.25,0.25) rectangle (0.25,0.75);
            \draw (0.25,0.25) rectangle (0.75,0.75);
            \draw (0.75,0.25) rectangle (1.25,0.75);
            \draw (1.25,0.25) rectangle (1.75,0.75);
            \draw (2,0.5) node {$\ldots$} ;
            \draw (2.25,0.25) rectangle (2.75,0.75);

            \draw (-0.25,0.75) rectangle (0.25,1.25);
            \draw (0.25,0.75) rectangle (0.75,1.25);
            \draw (0.75,0.75) rectangle (1.25,1.25);
            \draw (1.25,0.75) rectangle (1.75,1.25);
            \draw (2,1) node {$\ldots$} ;
            \draw (2.25,0.75) rectangle (2.75,1.25);

            \draw (-0.25,1.25) rectangle (0.25,1.75);
            \draw (0.25,1.25) rectangle (0.75,1.75);
            \draw (0.75,1.25) rectangle (1.25,1.75);
            \draw (1.25,1.25) rectangle (1.75,1.75);
            \draw (2,1.5) node {$\ldots$} ;
            \draw (2.25,1.25) rectangle (2.75,1.75);

            \draw (-0.25,2.25) rectangle (0.25,2.75);
            \draw (0.25,2.25) rectangle (0.75,2.75);
            \draw (0.75,2.25) rectangle (1.25,2.75);
            \draw (1.25,2.25) rectangle (1.75,2.75);
            \draw (2,2.5) node {$\ldots$} ;
            \draw (2.25,2.25) rectangle (2.75,2.75);
        \end{tikzpicture}
    \end{center}
    if we fill the diagram with numbers 
    \begin{center}
        \begin{tikzpicture}
            \draw (-0.25,-0.25) rectangle (0.25,0.25);
            \draw (0,0) node {\tiny $n$-1};
            \draw (0.25,-0.25) rectangle (0.75,0.25);
            \draw (0.5,0) node {\tiny $n$-2};
            \draw (0.75,-0.25) rectangle (1.25,0.25);
            \draw (1,0) node {\tiny $\ldots$};
            \draw (1.25,-0.25) rectangle (1.75,0.25);
            \draw (2,0) node {\tiny $\ldots$} ;
            \draw (2.25,-0.25) rectangle (2.75,0.25);
            \draw (2.51,0) node {\tiny $n$\tiny -$n_1$};

            \draw (-0.25,0.25) rectangle (0.25,0.75);
            \draw (0,0.5) node {\tiny $n$-2};
            \draw (-0.25,0.75) rectangle (0.25,1.25);
            \draw (0,1) node {\tiny $\vdots$};
            \draw (-0.25,1.25) rectangle (0.25,1.75);
            \draw (0,2) node {\tiny $\vdots$};
            \draw (0.5,2) node {\tiny $\vdots$};
            \draw (1,2) node {\tiny $\vdots$};
            \draw (1.5,2) node {\tiny $\vdots$};
            \draw (2,2) node {\tiny $\vdots$};
            \draw (2.5,2) node {\tiny $\vdots$};
            \draw (-0.25,2.25) rectangle (0.25,2.75);

            \draw (-0.25,0.25) rectangle (0.25,0.75);
            \draw (0.25,0.25) rectangle (0.75,0.75);
            \draw (0.75,0.25) rectangle (1.25,0.75);
            \draw (1.25,0.25) rectangle (1.75,0.75);
            \draw (2,0.5) node {\tiny $\ldots$} ;
            \draw (2.25,0.25) rectangle (2.75,0.75);

            \draw (-0.25,0.75) rectangle (0.25,1.25);
            \draw (0.25,0.75) rectangle (0.75,1.25);
            \draw (0.75,0.75) rectangle (1.25,1.25);
            \draw (1.25,0.75) rectangle (1.75,1.25);
            \draw (2,1) node {\tiny $\ldots$} ;
            \draw (2.25,0.75) rectangle (2.75,1.25);

            \draw (-0.25,1.25) rectangle (0.25,1.75);
            \draw (0.25,1.25) rectangle (0.75,1.75);
            \draw (0.75,1.25) rectangle (1.25,1.75);
            \draw (1.25,1.25) rectangle (1.75,1.75);
            \draw (2,1.5) node {\tiny $\ldots$} ;
            \draw (2.25,1.25) rectangle (2.75,1.75);

            \draw (-0.25,2.25) rectangle (0.25,2.75);
            \draw (0.25,2.25) rectangle (0.75,2.75);
            \draw (0.75,2.25) rectangle (1.25,2.75);
            \draw (1.25,2.25) rectangle (1.75,2.75);
            \draw (2,2.5) node {\tiny $\ldots$} ;
            \draw (2.25,2.25) rectangle (2.75,2.75);
            \draw (0,2.5) node {\tiny $n$-$n_2$};
            \draw (0.5,2.5) node {\tiny $\ldots$};
            \draw (1,2.5) node {\tiny $\ldots$};
            \draw (2.5,2.5) node {\tiny $1$};
            \draw (2.5,0.5) node {\tiny $\vdots$};
        \end{tikzpicture}.
    \end{center}
    For the Weyl group element $w_{[n_1-k_1,n_1-k_2,\ldots,n_1-k_{n_2}]}$, setting the partition $[l_{n_1},\ldots,l_1]$ as the transpose of the partition $[k_{n_2},\ldots,k_1]$ (adding zeros to make the length of the partition equal to $n_1$),
    then the inverse of the element $w_{[n_1-k_1,n_1-k_2,\ldots,n_1-k_{n_2}]}$ is equal to the element
    \[
        w'_{[n_2-l_1,\ldots,n_2-l_{n_1}]} = c_{n-n_2,n-1-l_{1}}c_{n-n_2-1,n-2-l_{2}}\ldots c_{1,n-n_1-l_{n_1}}
    \]
    which is the word constructed by reading the filled Young subdiagram of the rectangular Young diagram from top to bottom along each column, starting from the left-most column to the right-most column.
\end{lemma}
\begin{proof}
    By direct enumeration, in the word $a_1 a_2\ldots a_{n_1+n_2}$, the number of elements on the left of 1 is equal to the total number of rows of length $n_1$ in the Young diagram $[n_1-k_1,\ldots,n_1-k_{n_2}]$, which is equal to $n_2 - l_{n_1}$. The cycle to move 1 back to the first entry is thus $c_{1,n_2-l_{n_1}}$. In general, the number of elements from $(n_1+1)\ldots(n_1+n_2)$ on the left of $s\in \{1,\ldots,n_1\}$ is equal to the number of rows of length $\geq n_1-s+1$, which is equal to $n_2 - l_{n_1-s+1}$. Thus, the cycle to move $s$ back to the first entry is $c_{s,n_2-l_{n_1-s+1}+s-1}$. 
\end{proof}

\begin{lemma}\label{lemma1}
    \begin{enumerate}
        \item The elements of the coset $P_{[n_1,n_2]}\backslash G/B$ can be represented by the collection of Young subdiagrams of the rectangular Young diagram of height $n_2$ and width $n_1$.
        \item If we fill the above rectangular Young diagram with numbers:
        \begin{center}
            \begin{tikzpicture}
                \draw (-0.25,-0.25) rectangle (0.25,0.25);
                \draw (0,0) node {\tiny $n$-1};
                \draw (0.25,-0.25) rectangle (0.75,0.25);
                \draw (0.5,0) node {\tiny $n$-2};
                \draw (0.75,-0.25) rectangle (1.25,0.25);
                \draw (1,0) node {\tiny $\ldots$};
                \draw (1.25,-0.25) rectangle (1.75,0.25);
                \draw (2,0) node {\tiny $\ldots$} ;
                \draw (2.25,-0.25) rectangle (2.75,0.25);
                \draw (2.51,0) node {\tiny $n$\tiny -$n_1$};
    
                \draw (-0.25,0.25) rectangle (0.25,0.75);
                \draw (0,0.5) node {\tiny $n$-2};
                \draw (-0.25,0.75) rectangle (0.25,1.25);
                \draw (0,1) node {\tiny $\vdots$};
                \draw (-0.25,1.25) rectangle (0.25,1.75);
                \draw (0,2) node {\tiny $\vdots$};
                \draw (0.5,2) node {\tiny $\vdots$};
                \draw (1,2) node {\tiny $\vdots$};
                \draw (1.5,2) node {\tiny $\vdots$};
                \draw (2,2) node {\tiny $\vdots$};
                \draw (2.5,2) node {\tiny $\vdots$};
                \draw (-0.25,2.25) rectangle (0.25,2.75);
    
                \draw (-0.25,0.25) rectangle (0.25,0.75);
                \draw (0.25,0.25) rectangle (0.75,0.75);
                \draw (0.75,0.25) rectangle (1.25,0.75);
                \draw (1.25,0.25) rectangle (1.75,0.75);
                \draw (2,0.5) node {\tiny $\ldots$} ;
                \draw (2.25,0.25) rectangle (2.75,0.75);
    
                \draw (-0.25,0.75) rectangle (0.25,1.25);
                \draw (0.25,0.75) rectangle (0.75,1.25);
                \draw (0.75,0.75) rectangle (1.25,1.25);
                \draw (1.25,0.75) rectangle (1.75,1.25);
                \draw (2,1) node {\tiny $\ldots$} ;
                \draw (2.25,0.75) rectangle (2.75,1.25);
    
                \draw (-0.25,1.25) rectangle (0.25,1.75);
                \draw (0.25,1.25) rectangle (0.75,1.75);
                \draw (0.75,1.25) rectangle (1.25,1.75);
                \draw (1.25,1.25) rectangle (1.75,1.75);
                \draw (2,1.5) node {\tiny $\ldots$} ;
                \draw (2.25,1.25) rectangle (2.75,1.75);
    
                \draw (-0.25,2.25) rectangle (0.25,2.75);
                \draw (0.25,2.25) rectangle (0.75,2.75);
                \draw (0.75,2.25) rectangle (1.25,2.75);
                \draw (1.25,2.25) rectangle (1.75,2.75);
                \draw (2,2.5) node {\tiny $\ldots$} ;
                \draw (2.25,2.25) rectangle (2.75,2.75);
                \draw (0,2.5) node {\tiny $n$-$n_2$};
                \draw (0.5,2.5) node {\tiny $\ldots$};
                \draw (1,2.5) node {\tiny $\ldots$};
                \draw (2.5,2.5) node {\tiny $1$};
                \draw (2.5,0.5) node {\tiny $\vdots$};
            \end{tikzpicture}.
        \end{center}
        the double coset $P_{[n_1,n_2]}\backslash G/Q_{[n-m,1^m]}$ corresponds to the collection of Young subdiagrams with no columns filled with numbers only from $\{1,\ldots,n-m-1\}$.
        \item In a word $a_1a_2\ldots a_{n_1+n_2}$ representing a permutation $w$, we use the light gray intervals $v_i$ to represent consecutive intervals of elements from $(n_1+1)\ldots (n_1+n_2)$, and the dark gray intervals $u_i$ as consecutive intervals of elements coming from $12\ldots n_1$ (it is possible for $u_0$ to be empty).
        \begin{center}
            \begin{tikzpicture}[line width=10pt]
                \draw (-0.5,0.5) node {$u_0$};
                \draw[opacity = 0.5] (-1,0) --(0,0);
                \draw (0.5,0.5) node {$v_1$};
                \draw[opacity = 0.2] (0,0) --(1,0);
                \draw (1.5,0.5) node {$u_1$};
                \draw[opacity = 0.5] (1,0) --(2,0);
                \draw (2.5,0.5) node {$v_2$};
                \draw[opacity = 0.2] (2,0) --(3,0);
                \draw (3.5,0.5) node {$u_2$};
                \draw[opacity = 0.5] (3,0) --(4,0);
                \draw (5,0.5) node {$\ldots$};
                \draw[opacity = 0.1] (4,0) --(6,0);
                \draw (6.5,0.5) node {$u_i$};
                \draw[opacity = 0.5] (6,0) --(7,0);
                \draw (7.5,0.5) node {$v_{i+1}$};
                \draw[opacity = 0.2] (7,0) --(8,0);
                \draw (8.5,0.5) node {$\ldots$};
                \draw[opacity = 0.1] (8,0) --(9,0);
            \end{tikzpicture}.
        \end{center}
        The sequence $\left(|v_1|, |v_1|+|v_2|+|u_1|,\ldots,\sum_{j=1}^{i+1}|v_j|+\sum_{j=1}^i|u_j|\right)$ is equal to the sequence of numbers at the pivot positions of the filled subdiagram. The starting index of each interval $u_1,u_2,\ldots,u_i$ is given by the number \emph{above} each pivot position in the first row.
    \end{enumerate}
\end{lemma}
\begin{proof}
    There is a bijection between double cosets $P_{[n_1,n_2]}\backslash G/B$ and $B\backslash G/P_{[n_1,n_2]}$ given by $w\mapsto w^{-1}$. Also, by representing these elements as permutations, there is a bijection between the double coset $B\backslash G/P_{[n_1,n_2]}$ and the set of interlacings between two strings $12\ldots n_1$ and $(n_1+1)\ldots (n_1+n_2)$. The first part of this lemma follows from Lemma \ref{lemmabij1}. Under the bijection $w\mapsto w^{-1}$ between $P_{[n_1,n_2]}\backslash G/B$ and $B\backslash G/P_{[n_1,n_2]}$, the representatives surviving in the double coset $P_{[n_1,n_2]}\backslash G/Q_{[n-m,1^m]}$ are those elements $w'_{[n_2-l_1,\ldots,n_2-l_{n_1}]}$ from Lemma \ref{lemmainverse} with the right-most cycles not in $S_{n-m}$, which correspond to the filled diagrams with no column filled with numbers entirely from $\{1,\ldots,n-m\}$. The third part of the lemma follows from the proof of Lemma \ref{lemmainverse}. Since in the cycle representation of $w_{[n_1-k_1,\ldots,n_1-k_{n_2}]}$, each cycle $c_{s+k_s,n_1+s-1}$ moves the element $n_1+s$ to the $k_s+s$-th position of the word $a_1\ldots a_{n_1+n_2}$, if two adjacent rows $s-1$ and $s$ have the same length, the action by $w_{[n_1-k_1,\ldots,n_1-k_{n_2}]}$ will move these two rows to the $(k_{s-1}+s-1)$ and $(k_{s-1}+s)$-th position, respectively. If the numbers of columns of the same length are listed in the sequence $(b_1,\ldots,b_r)$ which corresponds to the sequence of pivot positions $(k_{s_1},\ldots,k_{s_r})$, then the lengths of each $v_i$ is equal to $b_{i}$, and the sequence $(k_{s_1},\ldots,k_{s_r})$ marks the starting indices of every $u_i$.
\end{proof}
The next lemma will be used to describe the Weyl group representatives corresponding to nonvanishing summands in (\ref{geomlemma}). We will put a partial order given by the lexicographic order of the coordinates $(a_1,a_2,\ldots,a_{n-1})$ of any root $a_1\alpha_1+\ldots +a_{n-1}\alpha_{n-1}$ on the set of positive roots of $GL_n$.
\begin{lemma}\label{lemmatail}
    For any $w\in P_{[n_1,n_2]}\backslash G/Q_{[n-s,1^s]}$, and for any $0\leq s\leq n$, there are no elements from \[\{w\alpha_{n-s+1},\ldots,w\alpha_{n-1}\}\] satisfy $\geq {\alpha_{n-s}}$ only if the Young diagram corresponding to $w$ is a rectangular diagram of height $n_2$.
\end{lemma}
\begin{proof}
    The set of elements \[\{w\alpha_{n-s+1},\ldots,w\alpha_{n-1}\}\] does not intersect the cone $\alpha_{n-s}$ if and only if in the tail of last $s$ elements of the corresponding word $v_0u_1v_1\ldots$, any $i\in\{1,\ldots,n_1\}$ and $j\in\{n_1+1,\ldots,n_1+n_2\}$ appearing in this tail satisfy $w(i)>w(j)$. Therefore, for any $s$, the tail of the string corresponding to $w$ must take the form of one of the following cases (if every entry is a positive integer):
    \begin{align*}
        t_0 = \{n_1-s+1,n_1-s+2,&\ldots ,n_1-1,n_1\}\\
        t_1 = \{n_1+n_2,n_1-s+2,&\ldots ,n_1-1,n_1\}\\
        t_2 = \{n_1+n_2-1,n,&\ldots, n_1-1,n_1\}\\
        &\ldots\\
        t_{s} = \{n-s+1,n-s+2,&\ldots ,n\}.
    \end{align*}
    Note that if $s > n_2$, the tails listed above exist only up to $t_{n_2}$, if $s > n_1$, the list of tails start from $t_{s-n_1}$.
    Denoting the total number of elements in the tail coming from $\{n_1+1,\ldots,n_1+n_2\}$ by $s_1$, and the total number of elements in the tail coming from $\{1,\ldots,n_1\}$ by $s_2$, by the third part of Lemma \ref{lemma1}, the pivot positions of the filled Young diagram corresponding to $w$ must satisfy the following two properties:
    \begin{enumerate}
        \item The bottom-most pivot box has content $n-s_2$, corresponding to the index $n_1-s_2$ on the first row of the Young diagram,
        \item The content of the next pivot box is an integer $< n-s$.
    \end{enumerate}
    However, in the double coset $P_{[n_1,n_2]}\backslash G/Q_{[n-s,1^s]}$, columns entirely consisting of boxes with contents $<n-s$ are not permitted. Therefore, only one pivot box is allowed, and the only situation which also allows (1) is when the Young diagram is a rectangular Young diagram of height $n_2$.
\end{proof}

Summarizing the two lemmas above, denoting by $\mathbb{W}_{n_1,n_2}^s$ the subset consisting of the representatives of the double coset $P_{[n_1,n_2]}\backslash G/ Q_{[n-s,1^{s}]}$ corresponding to the nonvanishing summands of (\ref{geomlemma}) are those which satisfy the condition:
\[
    \Ad(w)\left(\bigoplus_{n_1+1\leq i\leq n-1}\frakg_{\alpha_{i}}\right) \cap \bigoplus_{\alpha\geq\alpha_{n_1}}\frakg_{\alpha}\neq \left\{0\right\}.
\]
The set $\mathbb{W}_{n_1,n_2}^s$ can be described by the following lemma:
\begin{lemma}\label{cosetlemma}
    For $n=n_1+n_2$ and $\nu = n_1-n_2$, denoting the Weyl group element corresponding to a rectangular diagram of width $j$ and height $n_2$ by $w(j)$, the subset $\mathbb{W}_{n_1,n_2}^s$ of representatives of the double coset $P_{[n_1,n_2]}\backslash G/ Q_{[n-m,1^{s}]}$ are
    \begin{enumerate}
        \item For $s\leq \min(n_1,n_2)$, the representatives are \[w(0),w(1),\ldots,w(s),\]
        \item For $\min(n_1,n_2) < s\leq \max(n_1,n_2)$, the representatives are \[w(0),w(1),\ldots,w(\min(n_1,n_2)),\]
        \item For $s > \max(n_1,n_2)$, the representatives are \[w(s-\max(n_1,n_2)),w(m-\max(n_1,n_2)+1),\ldots,w(\min(n_1,n_2)).\]
    \end{enumerate}
The parabolic subgroups $M_{[n-s]}\cap w^{-1}P_{[n_1,n_2]}w\subset GL_{n-s}$ for $w\in\mathbb{W}_{n_1,n_2}^s $ have Levi subgroups listed by the partitions $s_1+s_2 = s$ such that $[n_1-s_1,n_2-s_2]$ is a partition of $n-s$.
\end{lemma}
\begin{proof}
The lemma follows from direct enumeration of the rectangular diagrams. If $s\leq \min(n_1,n_2)$, in the proof of Lemma \ref{lemmatail}, since no columns completely consisting of elements from $\{1,2,\ldots, n-s-1\}$ are allowed, the largest rectangular diagram with no column completely filled with numbers in ${1,2,\ldots,n-s}$ is the one corresponding to the partition $[s^{n_2}]$. If $n_1 \leq n_2$, in the proof of Lemma \ref{lemmatail}, we can learn that if $n_1 < s \leq n_2$, the list of tails range from $t_{s-n_1}$ up to $t_{s}$. But if $s > n_2$, the list of tails ranges from $t_{s-n_1}$ up to $t_{n_2}$. 
When $n_1 > n_2$, if $n_2 < s \leq n_1$, the tails exist from $t_0$ to $t_{n_2}$, while if $s > n_1$, the list of tails ranges from $t_{s-n_1}$ up to $t_{n_2}$. In either case, the total number of representatives agrees with the result in part (2) and (3) of this lemma, and the representatives match with the representatives described in Lemma \ref{lemmatail}. A rectangular diagram of width $j$ sends $n_1+t$ to the position $n_1+t-j$ of the word corresponding to the permutation. As in the proof of Lemma \ref{lemmatail}, $(s_1,s_2) = (s-j,j)$, and if $s,j$ satisfies the conditions as described in the statement of the lemma, $[n_1-s+j,s-j]$ and $[n_2-j,j]$ are partitions of the integers $n_1$ and $n_2$, respectively. The block sizes of the Levi subgroup $w^{-1}P_{[n_1,n_2]}w$ are given by the partition $[n_1-s+j,s-j,n_2-j,j]$, and the sizes of Levi blocks of the parabolic subgroup $M_{[n-s]}\cap w^{-1}P_{[n_1,n_2]}w\subset GL_{n-s}$ are given by the partition $[n_1-s+j, n_2-j]$.
\end{proof}
The reasoning in this section can be summarized in the following corollary:
\begin{corollary}\label{advgeomlemma}
    The coset $W_{M,M'}$ in the formula (\ref{geomlemma}) can be replaced by $\mathbb{W}_{n_1,n_2}^s$:
    \begin{align*}
        D^{(s)} E(\tilde{\varphi},g)
        =\sum_{w\in \mathbb{W}_{n_1,n_2}^s}\mathcal{W}_s\left(\sum_{m'\in M_{[n-s]}\cap w^{-1}P_{[n_1,n_2]}w\backslash M_{[n-s]}(k)}\mathcal{A}_{w}\left(c_{w^{-1}M_{[n-s]}w\cap P_{[n_1,n_2]}}\tilde\varphi\right)(wm'g)\right).
    \end{align*}
\end{corollary}

\section{Degenerate Whittaker Coefficients and Derivatives}
\subsection{Overview}
The set of nilpotent orbits allowing nonzero Fourier coefficients for an automorphic representation, as well as its local analogue, is a subject studied by many mathematicians. For the case of $GL(n)$, we summarize timeline of results and applications as follows:
\begin{itemize}
    \item In \cite[Theorem 3.1 and 4.5]{shalika1974multiplicity}, Shalika proved that the local and global generic Whittaker model, whenever exists, has multiplicity 1. In particular, Whittaker model always exists for cuspidal automorphic representation. Based on this, \cite{shalika1974multiplicity} and \cite{piatetski1979multiplicity} proved the local and global multiplicity-one theorem for $GL(n)$. The existence of global Whittaker model is also important in the construction of automorphic $L$-functions with the Langlands-Shahidi method.
    \item In \cite[Proposition 5.3]{ginzburg2006certain}, Ginzburg proved that the maximal orbit allowing a nonzero Fourier coefficient for the generalized Speh representation (as a residue of cuspidal Eisenstein series) $\Delta(\tau,n)$ is $[a^n]$ by showing the vanishing of the induced local Whittaker functional. In \cite[Theorem 2.5.4-2.5.6]{liu2013fourier}, the same result was proven following a global argument based on the \emph{root exchange} technique. A recipe for the \emph{root exchange} technique is available in Section 7.1 of \cite{ginzburg2011descent}, and a systematic reduction method based on the root exchange technique is developed by Gourevitch et al. in \cite{gourevitch2022reduction}. A similar global method also appears in \cite{cai2018fourier} for the orbit of the top Fourier coefficient of degenerate Eisenstein series of $GL(n)$, and in \cite{liu2020top} for the orbit corresponding to the top Fourier coefficient of isobaric sum representations.
\end{itemize}
In this section, we will provide a global, inductive proof to the vanishing of the Fourier coefficients associated with larger orbits, as well as an argument for the Eulerianity of certain Fourier-Whittaker coefficients.

\subsection{Whittaker Support and Wave-front Set of Automorphic Forms}
This section follows the notations and terminologies in \cite{gourevitch2022reduction}. For any $\mathbb{Q}$-semisimple element $S\in\frakg$, denoting by $\frakg^S_\mu$ the $\mu$-eigenspace of the adjoint action $\mathrm{ad}(S)$ on $\frakg$,
\begin{itemize}
    \item A \emph{Whittaker pair} $(S,\phi)\in\frakg\times\frakg^*$ contains a $\mathbb{Q}$-semisimple element $S$ and a nilpotent element $\phi\in (\frakg^*)^S_{-2}$. The element $\phi$ can be constructed by pairing the Killing form of $\frakg$ with a nilpotent element $f_\phi \in\frakg^S_{-2}$.
    \item A Whittaker pair $(h,\phi)$ is \emph{neutral} if it can be completed to a standard $\mathfrak{sl}_2$-triple.
    \item A Whittaker pair $(S,\phi)$ is \emph{standard} if the nilpotent subalgebra
    \[
        \frakn_{(S,\phi)} = \frakg^S_{>1}\oplus \left(\frakg^S_{1}\cap \frakg_\phi\right),
    \]
    where $\frakg_\phi$ is the stabilizer of $\phi$, is the nilpotent radical of the minimal parabolic subalgebra.
    \item A nilpotent element $\phi\in\frakg^*$ is $k$-distinguished if the corresponding nilpotent element $f_\phi\in\frakg$ does not belong to a proper $k$-Levi subalgebra of $\frakg$. For a neutral Whittaker pair $(h,\phi)$, the Whittaker pair $(h+Z,\phi)$, with $Z$ a semisimple element centralizing the pair $(h,\phi)$ and $\phi$ a $k$-distinguished nilpotent element, is called \emph{Levi-distinguished} if
    \begin{align*}
        \frakg_{>1}^{h+Z} &= \frakg_{\geq 2}^{h+Z} = \frakg_{>0}^{Z}\oplus \frakl_{\geq 2}^{h}\\
        \frakg_{1}^{h+Z} &= \frakl_{1}^h.
    \end{align*}
    where $\mathfrak{l}$ is the centralizer of $Z$. In particular, for the case of $\mathfrak{gl}_n$ and $f_\phi$ a usual Jordan form, the Whittaker pair $(H, \phi)$ with $H = \mathrm{diag}\left(n-1,n-3,\ldots,3-n, 1-n\right)$ is a Levi-distinguished Whittaker pair.
\end{itemize}
For any two Whittaker pairs $(H,\phi)$ and $(S,\phi)$ containing the same nilpotent element $\phi$, $(H,\phi)$ \emph{dominates} $(S,\phi)$ (denoted by $(H,\phi)\prec (S,\phi)$) if $H$ and $S$ commute, and
\[
    \frakg_\phi\cap \frakg_{\geq 1}^H \subset \frakg_{\geq 0}^{S-H}.
\]
In \cite[Corollary 3.2.2 and Proposition 3.2.3]{gourevitch2022reduction}, it is shown that any Whittaker pair $(H,\phi)$ is dominated by a neutral Whittaker pair with the same nilpotent element, and there exists a $Z$ in the centralizer of $(H,\phi)$ such that $(H,\phi)\prec (H+Z,\phi)$ with $(H+Z,\phi)$ standard.\\

We can define the Whittaker coefficient $\mathcal{W}_{(H,\phi)}(F)$ of any automorphic function $F$ corresponding to any Whittaker pair $(H,\phi)$ as the following integral
\[
    \mathcal{W}_{H,\phi}(F) = \int_{N_{S,\phi}} F(ng)\overline{\phi(n)}dn
\]
where $N_{S,\phi}$ is the analytic subgroup with Lie algebra $\frakn_{S,\phi}$ the nilpotent radical of the symplectic form \[
    \omega_\phi\left(X,Y\right) = \phi([X,Y])
\]
restricted to the nilpotent subalgebra $\fraku = \frakg_{\geq 1}^S$.\\

A nilpotent element $\phi$ is in the \emph{Whittaker support} $\mathrm{WS}(F)$ of any automorphic function $F$ if there exists a Whittaker support $(S,\phi)$ such that $\mathcal{W}_{H,\phi}(F)\neq 0$, and for any Whittaker pair $(H',\psi)$ with $\psi$ not in the closure of the nilpotent orbit of $\phi$ we have $\mathcal{W}_{H',\psi}(F)= 0$.\\

The following theorem from \cite[Theorem B]{gourevitch2022reduction} provides a method to calculate the Whittaker support of an automorphic function.
\begin{theorem}\label{gourevitchtheorem}
    If $\phi$ is in $\mathrm{WS}(F)$ and two Whittaker pairs satisfy $(H,\phi)\prec (S,\phi)$, then there exists an integral operator $\mathcal{M}_H^S$ such that
    \[
        \mathcal{W}_{H,\phi}[F] = \mathcal{M}_H^S\left(\mathcal{W}_{S,\phi}[F]\right).
    \]
\end{theorem}
By this theorem, and \cite[Corollary 3.2.2 and Proposition 3.2.3]{gourevitch2022reduction}, the vanishing of Whittaker coefficient corresponding to a standard Whittaker pair implies the vanishing of any Whittaker coefficient. 

\subsection{Composition of Derivative Operators}
We can compose multiple derivative operators.
\begin{proposition}
    For any two derivative operators 
    \begin{align*}
        D^{(m_2)}: \mathcal{A}(G_n(k)\backslash G_n(\mathbb{A}))&\longrightarrow \mathcal{A}(G_{n-m_2}(k)U_{[(n-m_2) 1^{m_2}]}(\mathbb{A})\backslash G_n(\mathbb{A}))\\
        D^{(m_1)}: \mathcal{A}(G_{n-m_2}(k)\backslash G_{n-m_2}(\mathbb{A}))&\longrightarrow \mathcal{A}(G_{n-m_2-m_1}(k)U_{[(n-m_2-m_1) 1^{m_1}]}(\mathbb{A})\backslash G_{n-m_2}(\mathbb{A})),
    \end{align*}
    if we define the composition of two operators as
    \[
        D^{(m_1)}\circ D^{(m_2)}(f) = D^{(m_1)}\left(\left(D^{(m_2)}f\right)\vert_{G_{n-m_2}}\right),
    \]
    then this operator $D^{(m_1)}\circ D^{(m_2)}$ is equal to the Whittaker coefficient $\mathcal{W}_{U_{[(n-m_1-m_2)1^{m_1+m_2}]},\psi}$ along the unipotent subgroup 
    \[
        U_{[(n-m_1-m_2)1^{m_1+m_2}]} = \left\{\begin{pmatrix}
            I_{n-m_1-m_2} & v\\
            0 & u_{m_1+m_2}
        \end{pmatrix}\mid v\in \mathrm{Mat}_{n-m_1-m_2,m_1+m_2}(\bbA), u_{m_1+m_2}\in U_{[m_1+m_2]}\right\}
    \]
    with $u_{m_1+m_2}$ a strictly upper-triangular matrix of size $m_1+m_2$, and the character $\psi$ given by the partition $(m_1,m_2)$ on $GL_{m_1+m_2}$:
    \begin{align*}
        \psi\begin{pmatrix}
            I_{n-m_1-m_2} & v\\
            0 & u_{m_1+m_2}
        \end{pmatrix} = \psi_{[m_1,m_2]}(u_{m_1+m_2}).
    \end{align*}
\end{proposition}
\begin{proof}
    For any automorphic form $f\in \mathcal{A}(G_n(k)\backslash G_n(\mathbb{A}))$, the function $D^{(m_2)}f$ can be restricted to the subgroup $GL_{n-m_2}$ on the top-left corner. After composing with $D^{(m_1)}$, by the matrix decomposition
    \[
        \begin{pmatrix}
            I_{n-m_1-m_2} & v\\
            0 & u_{m_1+m_2}
        \end{pmatrix} = \begin{pmatrix}
            I_{n-m_1-m_2} & 0& v_3\\
            0 & I_{m_1} & v_2\\
            0 & 0 & u_{m_2}
        \end{pmatrix}\begin{pmatrix}
            I_{n-m_1-m_2} & v_1& 0\\
            0 & u_{m_1} & 0\\
            0 & 0 & 1
        \end{pmatrix},
    \]
    thus we can decompose the integral as
    \begin{align*}
        &\int_{[U_{[(n-m_1-m_2)1^{m_1+m_2}]}]}f\left(\begin{pmatrix}
            I_{n-m_1-m_2} & v\\
            0 & u_{m_1+m_2}
        \end{pmatrix}g\right)\psi_{[m_1,m_2]}\begin{pmatrix}
            I_{n-m_1-m_2} & v\\
            0 & u_{m_1+m_2}
        \end{pmatrix} du_{m_1+m_2}dv\\
        &= \int_{[U_{[(n-m_1-m_2)1^{m_1}]}]}\int_{[U_{[(n-m_2)1^{m_2}]}]}f\left(\begin{pmatrix}
            I_{n-m_1-m_2} & 0& v_3\\
            0 & I_{m_1} & v_2\\
            0 & 0 & u_{m_2}
        \end{pmatrix}\begin{pmatrix}
            I_{n-m_1-m_2} & v_1& 0\\
            0 & u_{m_1} & 0\\
            0 & 0 & 1
        \end{pmatrix}g\right)\times\\
        &\;\;\;\;\;\;\;\;\;\;\;\;\psi_{m_2}\begin{pmatrix}
            I_{n-m_1-m_2} & 0& v_3\\
            0 & I_{m_1} & v_2\\
            0 & 0 & u_{m_2}
        \end{pmatrix}\psi_{m_1}\begin{pmatrix}
            I_{n-m_1-m_2} & v_1& 0\\
            0 & u_{m_1} & 0\\
            0 & 0 & 1
        \end{pmatrix}du_{m_2}du_{m_1}dv_1dv_2dv_3.
    \end{align*}
\end{proof}
Using this proposition, for any Whittaker pair $(H,\phi)$ with $H = \diag(n-1,n-3,\ldots,3-n,1-n)$ and $\phi$ a nilpotent element with Jordan type $(\lambda_1,\ldots,\lambda_r)$, the corresponding Fourier-Whittaker coefficient operator can be identified with a composition of derivative operators:
\[
    \mathcal{W}_{H,\phi} = D^{(\lambda_1)}\circ \ldots \circ D^{(\lambda_r)},
\]
which is the degenerate Whittaker coefficient along the maximal unipotent radical with respect to the character
\begin{align*}
    \psi_{\alpha}\left(
        \begin{smallmatrix}
            1 & u_1 & \ldots & & &\\
              & 1   & u_2 & \ldots & &\\
              & & 1 & u_3 & \ldots &&\\
              & & & & \ldots&&\\
              & & & & &1 & u_{n-1}\\
              & & & & & & 1
        \end{smallmatrix}
    \right) = &\psi(u_1+\ldots+u_{\alpha_1-1})\psi(u_{\alpha_1+1}+\ldots+u_{\alpha_1+\alpha_2-1})\ldots\\
    &\psi(u_{\alpha_1+\ldots+\alpha_{s-1}+1}+\ldots+u_{\alpha_1+\ldots+\alpha_{s}-1}).
\end{align*}

\subsection{A First Example: Degenerate Eisenstein Series}
For an Eisenstein series attached to the parabolic induction from a trivial representations on the Levi subgroup $L_{[n_1,n_2]}$:
\[
    E_{\underline{s}}(\cdot): \mathrm{Ind}^{G_{n}}_{P_{[n_1,n_2]}}\left(|\cdot|^{s_1}\otimes |\cdot|^{s_2}\right)\rightarrow \mathcal{A}_n,
\]
we would like to understand the derivative $D^{(m)}E_{\underline{s}}(\phi)$ with $\phi$ the spherical vector of the principal series $\mathrm{Ind}^{G_{n}}_{P_{[n_1,n_2]}}\left(|\cdot|^{s_1}\otimes |\cdot|^{s_2}\right)$. By Corollary \ref{advgeomlemma}, the derivative $D^{(m)}E_{\underline{s}}(\phi)$ can be written as a sum
\begin{align}\label{iwintegral}
    D^{(m)} E_{\underline{s}}(\phi) = \sum_{w\in \mathbb{W}^m_{n_1,n_2}} I_w(\underline{s},\phi),
\end{align}
in which each individual integral $I_w(\underline{s}, \phi)$ is given by the composition of three operators:
\[
    I_w(\underline{s},\phi) = \mathcal{W}_{m}\circ\mathcal{A}_w\circ c_{[n_1-m_1,m_1],[n_2-m_2,m_2]}\left(\phi\right),
\]
with the correspondence of $w$ and the pair of integers $(m_1,m_2)$ is given in Lemma \ref{cosetlemma}. It is possible that many terms $I_w(\underline{s},\phi)$ in the sum will vanish. The vanishing of such terms is described in the following proposition:
\begin{proposition}\label{propcosetdegen}
    All terms $I_w(\underline{s},\phi)$ vanish except for the terms corresponding to the following Weyl group elements: 
    \begin{enumerate}
        \item When $m=1$: 
        \begin{align*}
            w_{[0,1]} &= \mathrm{id},\\
            w_{[1,0]} &= \left\{1,\ldots, n_1-1, n_1+1,\ldots,n_1+n_2,n_1\right\}.
        \end{align*}
        \item When $m=2$: \begin{align*}w_{[1,1]}=\left\{1,\ldots, n_1-1,n_1+1,\ldots,n_1+n_2-1,n_1,n_1+n_2,\right\}.\end{align*}
    \end{enumerate}
\end{proposition}
\begin{proof}
    We prove this proposition by induction on the rank of the group $GL_n$. Setting $n = 2$, the case $(n_1,n_2) = (2,0)$ or $(0,2)$ corresponds to the situation when the representation is a character with no induction involved. In this case, $D^{(m)}E_{\underline{s}}(\phi)$ is simply an integral along the unipotent radical, and is nonzero only if $m = 1$. For $(n_1,n_2) = (1,1)$, the only nonvanishing term in $D^{(m)}E_{\underline{s}}(\phi)$ corresponds to the cases $(m_1,m_2) = (1,0)$, $(0,1)$ and $(1,1)$. By (\ref{geomlemma}), the case $(m_1,m_2) = (1,0)$ or $(0,1)$ corrresponds to the constant term of the $GL_2$ Eisenstein series, and the case $(1,1)$ corresponds to the calculation of the Whittaker coefficient of the Eisenstein series.\\
    \indent Assuming that the proposition holds for all lower-rank cases, by Lemma \ref{cosetlemma}, it suffices to study the vanishing of each $I_w(\underline{s},\phi)$ for every Weyl group element of the form
    \[ w_{(m_1,m_2)} = \left(1,\ldots,n_1-m_1;n_1+1,\ldots,n_1+n_2-m_2;n_1-m_1+1,\ldots,n_1;n_1+n_2-m_2+1,\ldots,n_1+n_2\right).\]
    Moreover, by (\ref{geomlemma}) again, after restricted to the Levi subgroup $M_{[n_1-m_1,n_2-m_2,m_1,m_2]}$, the term $I_w(\underline{s},\phi)$ is the product of an Eisenstein series $E_{\underline{s}}^{[n-m_1,n-m_2]}$ on $GL_{n-m}$ and the Whittaker coefficient $\mathcal{W}_{m_1+m_2}\left(E_{\underline{s}}^{[m_1,m_2]}\right)$, which is the generic Whittaker coefficient along the maximal unipotent subgroup of the Eisenstein series constructed from the degenerate principal series $\mathrm{Ind}_{P_{[m_1,m_2]}}^{G_{m}}\left(|\cdot|^{s_1}\otimes |\cdot|^{s_2}\right)$. By the induction hypothesis, this coefficient $I_w(\underline{s},\phi)$ can be nonzero only if $m_1,m_2\leq 1$.
\end{proof}
Therefore, after restricted to the subgroup $GL_{n-m}$ on the top-left corner, the vector $I_w(\phi,g)$ finds itself in the space of automorphic functions generated by an Eisenstein series
\[
    E^{[n_1-m_1,n_2-m_2]}(\underline{s},\phi;g) = \sum_{\gamma\in P_{[n_1-m_1,n_2-m_2]}(k)\backslash G_{n-m}(k)}\phi(\gamma g)
\]
on the subgroup $GL_{n-m}$. In order to connect this observation to the study of Whittaker support of Eisenstein series, we need the following lemma:
\begin{lemma}\label{lemmapermute}
    For any neutral Whittaker pair $(h,\phi_\lambda)$ with $\phi_\lambda$ given by the Killing pairing $\phi_\lambda = K(\cdot,f_\lambda)$ with a Jordan form $f_\lambda$, and another Whittaker pair $(h',\phi_\lambda')$ obtained from $(h,\phi_\lambda)$ by permuting its Jordan blocks, with $h,h'$ lying in the same Cartan subalgebra, then for any Whittaker pair $(H,\phi'_\lambda)$, we have $\mathcal{W}_{h,\phi_\lambda} = 0$ if $\mathcal{W}_{H,\phi_\lambda'} = 0$.
\end{lemma}
\begin{proof}
    The Whittaker coefficients $\mathcal{W}_{h,\phi_\lambda}$ and $\mathcal{W}_{h',\phi_\lambda'}$ are Weyl group conjugates with each other. The result of this lemma follows from the existence of the integral operator $\mathcal{M}_{h',H}'$, according to Theorem \ref{gourevitchtheorem}, such that
    \[
        \mathcal{W}_{h',\phi_\lambda'} = \mathcal{M}_{h',H}'\left(\mathcal{W}_{H,\phi_\lambda'}\right).
    \]
    Thus, $\mathcal{W}_{H,\phi_\lambda'} = 0$ implies $\mathcal{W}_{h',\phi_\lambda'} = 0$.
\end{proof}

Using the method developed in this section, we can reprove the following result on the Whittaker support of degenerate Eisenstein series, first proven by Cai in \cite{cai2018fourier}: 
\begin{corollary}
    For the degenerate Eisenstein series $E_{\underline{s}}$ induced from the trivial representation of the Levi subgroup $L_{[n_1,n_2]}$, the degenerate Whittaker coefficient $\mathcal{W}_{U,\phi_\lambda}\left(E_{\underline{s}}\right)$ can be nonzero only if the corresponding partition $\lambda$ contains only $1$ and $2$. 
\end{corollary}
\begin{proof}
    We prove this corollary by induction on $n$. According the proof Proposition \ref{propcosetdegen}, the base cases when $n = 1$ or $2$ are straightforward. By Proposition \ref{propcosetdegen}, if neither of $n_1,n_2$ is zero, and without loss of generality, assuming $n_1\geq n_2$,  the Eisenstein series $E_{\underline{s}}$ allows both $D^{(1)}$ and $D^{(2)}$. By Lemma \ref{advgeomlemma} and (\ref{iwintegral}), the derivative $D^{(1)}E_{\underline{s}}$ is the sum of two vectors
    \[
        E^{[n_1-1,n_2]}\boxtimes\chi_1 + E^{[n_1,n_2-1]}\boxtimes\chi_2
    \]
    with $\chi_1,\chi_2$ unramified characters of $GL_1$. By the induction hypothesis, the allowed Whittaker coefficient of each one of these two terms corresponds to the partition $(2^{n_2}1^{n_1-n_2-1})$ and $(2^{n_2-1}1^{n_1-n_2+1})$, respectively. Therefore, the top orbit allowing a degenerate Whittaker coefficient corresponds to the partition $(2^{n_2}1^{n_1-n_2-1})$, with only the first term surviving. 
    The partial order of a partition containing only 1 or 2 is determined by the total number of 2's in the sequence. However, for any partition $(2^t 1^{n-2t})$ with $t > n_2$, the Whittaker coefficient corresponding to the partition $(2^t 1^{n-2t-1})$ kills both $E^{[n_1-1,n_2]}$ and $E^{[n_1,n_2-1]}$. Therefore, by Theorem \ref{gourevitchtheorem}, for any Whittaker pair $(S,\phi_\lambda)$ with $\lambda = (2^t 1^{n-2t-1}), t>n_2$, the corresponding $\mathcal{W}_{S,\phi_\lambda}$ is zero.
\end{proof}
The original result of Cai in \cite{cai2018fourier} says that the \emph{Whittaker support} of this degenerate Eisenstein series is $[1^{n_1}]+[1^{n_2}]$. For the degenerate Eisenstein series induced from a parabolic subgroup with more than two Levi blocks, we can perform induction by stages and construct Eisenstein series on a bigger group from smaller groups.

\subsection{Generalized Speh Representation}
The following lemma for the constant term of the generalized Speh representation can be found in \cite{liu2013fourier}:
\begin{lemma}\label{baiyinglemma}
    The constant term along $U_{[n_1,n_2]}$ with $n_1+n_2=n$ of the residue $\mathrm{Res}_n E_n(f,\underline{\nu})$ can be realized as a section in $\Ind_{P_{[n_1,n_2]}}^{G_n}(\Delta(\tau,n_1)|\cdot|^{-\frac{n_2}{2}}\boxtimes \Delta(\tau,n_2)|\cdot|^{\frac{n_1}{2}})$.
\end{lemma}
Applying the derivative operator $D^{(ra)}$, which is the composition of a constant term operator $c_{P_{[(n-r)a,ra]}}$ and a generic Whittaker coefficient operator $\mathcal{W}_{ra}$. By Lemma \ref{baiyinglemma}, for any $f\in \Delta(\tau,n)$, $c_{P_{[(n-r)a,ra]}}f$ is a section of the induced representation $\Ind_{P_{[n-r,r]}}^{G}(\Delta(\tau,n-r)|\cdot|^{-\frac{r}{2}}\boxtimes \Delta(\tau,r)|\cdot|^{\frac{n-r}{2}})$. 
\begin{proposition}\label{spehlemma}
    For a generalized Speh representation $\Delta(\tau,n)$ realized as the residue of an Eisenstein series with $\tau$ a cuspidal automorphic representation of $GL_a$, the only allowed derivative is $D^{(a)}$.
\end{proposition}
\begin{proof}
    We can prove this proposition by induction on $n$. When $n=1$, since $\tau$ is a cuspidal automorphic representation of $GL_a$, any constant term operator $c_{P_{[n_1,n_2]}}$ with $n_1$ or $n_2$ nonzero will vanish. Therefore, the only nonzero derivative is $D^{(a)}$, which corresponds to calculating the generic Whittaker coefficient of any vector in the cuspidal automorphic representation $\tau$. Assuming the proposition holds for all $k\leq n$, by Lemma \ref{baiyinglemma}, the only possibly nonzero constant terms are of the form $c_{P_{[(n-r)a,ra]}}$. However, when $r > 1$, by Lemma \ref{baiyinglemma}, Corollary \ref{advgeomlemma} and the induction hypothesis, the Whittaker coefficient $\mathcal{W}_{ra}$ as in Corollary \ref{advgeomlemma} vanishes. Thus the only allowed derivative for any vector $f\in \Delta(\tau,n)$ is 
    $D^{(a)}$.
\end{proof}
This process can be iterated. The derivative operator $D^{(a)}$ of $\Delta(\tau,n)$ can be realized as a composition of two operators
\[
    \Delta(\tau,n)\xrightarrow{c_{[(n-1)a,a]}} \Ind_{P_{[(n-1)a,a]}}^{G_{na}}(\Delta(\tau,n-1)|\cdot|^{-\frac{1}{2}}\boxtimes \tau|\cdot|^{\frac{n-1}{2}})\xrightarrow{\mathcal{W}_{a}} \Ind_{P_{[(n-1)a,a]}}^{G_{na}}(\Delta(\tau,n-1)|\cdot|^{-\frac{1}{2}}\boxtimes |\cdot|^{\frac{n-1}{2}})
\]
As a result, the only nonzero degenerate Whittaker coefficients $\mathcal{W}_{U,\phi_\lambda}(f)$ for $f\in \Delta(\tau,n)$ are those with a partition $\lambda$ containing only $a$, as shown in \cite[Theorem 2.3.3]{liu2013fourier}. This result can also be proven by our method, as stated in the following lemma:
\begin{corollary}\label{spehwhittaker}
    The only nonzero degenerate Whittaker coefficient of the generalized Speh representation $\Delta(\tau,n)$ is attached to the orbit $[a^n]$.
\end{corollary}
\begin{proof}
    For any pure tensor $f\in\Delta(\tau,n)$, by Proposition \ref{spehlemma}, the only allowed derivative is $D^{(a)}$. When restricting $D^{(a)}f$ to the subgroup $GL_{(n-1)a}$ embedded in the top-left corner, according to the proof of Proposition \ref{spehlemma}, we can see that $D^{(a)}f$ belongs to the generalized Speh representation $\Delta(\tau,n-1)$. 
    For the Whittaker pair $(H,\phi)$ with $H = (n-1,\ldots,1-n)$ and $\phi$ corresponding to the orbit $[a^n]$, for any pure tensor $f$, the corresponding degenerate Fourier coefficient $\mathcal{W}_{H,\phi}(f)$ is equal to the derivative $D^{(a)}\circ \ldots \circ D^{(a)}f$.
\end{proof}

\subsection{Induction from Discrete Spectrum Data}
We will study the Whittaker support of the Eisenstein series constructed from the induced representation from generalized Speh representations on a maximal parabolic subgroup:
\[
    I(\tau_1,\tau_2;\underline{s}) = \Ind_{P_{[n_1,n_2]}}^{G_n}(\Delta(\tau_1,n_1)|\cdot|^{s_1}\boxtimes \Delta(\tau_2,n_2)|\cdot|^{s_2}).
\]
We can construct an Eisenstein series $E_{\underline{s}}(\tilde{\phi})$ for any vector $\phi \in I(\tau_1,\tau_2;\underline{s})$. In the following two cases, this Eisenstein series will exhibit no poles:
\begin{enumerate}
    \item $\tau_1\not\cong\tau_2$,
    \item $\tau_1\cong\tau_2\cong \tau$, with $s_1-s_2 \gg 0$.
\end{enumerate}
\subsubsection{Case $\tau_1\not\cong\tau_2$}
Assuming that $\tau_i$ is a cuspidal automorphic representation on $GL_{a_i}$, the generalized Speh representation $\Delta(\tau_i,n_i)$ is an automorphic representation for the group $GL_{a_in_i}$. By Corollary \ref{advgeomlemma},
\begin{align*}
    D^{(s)} E_{\underline{s}}(\tilde{\varphi},g)
    =\sum_{w\in \mathbb{W}_{n_1,n_2}^s}\mathcal{W}_s\left(\sum_{m'\in M_{[n-s]}\cap w^{-1}P_{[n_1,n_2]}w\backslash M_{[n-s]}(k)}\mathcal{A}_{w}\left(c_{w^{-1}M_{[n-s]}w\cap P_{[n_1,n_2]}}\tilde\varphi\right)(wm'g)\right),
\end{align*}
the permitted constant terms are those along the parabolic subgroup $P_{[a_1(n_1-r_1),a_1r_1]}\times P_{[a_2(n_2-r_2),a_2r_2]}$ of the subgroup $GL_{a_1n_1}\times GL_{a_2n_2}$ for permitted $r_1,r_2$. Therefore, for the $r$-th derivative, the permitted constant terms in the sum correspond to any block partition
\[
    [n_1,n_2]\rightsquigarrow [n_1-p_1,p_1;n_2-p_2,p_2]
\]
satisfying $p_1a_1 + p_2a_2 = r$. The operator $\mathcal{A}_w$ corresponds to the Weyl group element
\[
    w = (1,2,\ldots,n_1-r_1;n_1+1,\ldots,n_1+n_2-r_2;n_1-r_1+1,\ldots,n_1;n_1+n_2-r_2+1;n_1+n_2)
\]
that intertwines between the following two principal series:
\begin{align*}
    &I\left(\Delta(\tau_1,n_1-r_1)|\cdot|^{-\frac{r_1}{2}+s_1}\boxtimes \Delta(\tau_1,r_1)|\cdot|^{\frac{n_1-r_1}{2}+s_1}\boxtimes\Delta(\tau_2,n_2-r_2)|\cdot|^{-\frac{r_2}{2}+s_2}\boxtimes \Delta(\tau_2,r_2)|\cdot|^{\frac{n_2-r_2}{2}+s_2}\right)\\
    \xrightarrow{\mathcal{A}_w}&\\
    &I\left(\Delta(\tau_1,n_1-r_1)|\cdot|^{-\frac{r_1}{2}+s_1}\boxtimes\Delta(\tau_2,n_2-r_2)|\cdot|^{-\frac{r_2}{2}+s_2}\boxtimes \Delta(\tau_1,r_1)|\cdot|^{\frac{n_1-r_1}{2}+s_1}\boxtimes \Delta(\tau_2,r_2)|\cdot|^{\frac{n_2-r_2}{2}+s_2}\right).
\end{align*}
Since $\tau_1\not\cong\tau_2$, the Eisenstein series constructed from the induced representation $I(\Delta(\tau_1,r_1)|\cdot|^{\frac{n_1-r_1}{2}+s_1}\boxtimes\Delta(\tau_2,n_2-r_2)|\cdot|^{-\frac{r_2}{2}+s_2})$ does not have any pole in the region $\mathrm{Re}(s_1-s_2)\gg 0$, and the intertwining operator $\mathcal{A}_w$ does not introduce any extra vanishing. In this situation, the nonzero derivatives can be described in the following proposition:
\begin{proposition}\label{notequal}
    If $\tau_1\not\cong \tau_2$, the only possibly nonzero derivatives of $I(\tau_1,\tau_2;s)$ are $D^{(a_1)}, D^{(a_2)}$ and $D^{(a_1+a_2)}$.
\end{proposition}
\begin{proof}
    In the base cases when $(n_1,n_2) = (1,0)$ or $(n_1,n_2) = (0,1)$, the principal series $I(\tau_1,\tau_2;s)$ is a single cuspidal automorphic representation, and the only nonzero derivative is the highest derivative $D^{(a_1)}$ or $D^{(a_2)}$. If $(n_1,n_2) = (1,1)$, the principal series $I(\tau_1,\tau_2;s)$ is generic, and the derivative $D^{(a_1+a_2)}$ calculates the generic Whittaker coefficient of the corresponding Eisenstein series. For other derivatives, since the constant term along any parabolic subgroup $Q$ is zero unless $Q$ and $P_{[n_1,n_2]}$ are associate, there are only two other possibly nonzero derivatives $D^{(a_1)}$ and $D^{(a_2)}$. Assuming the induction hypothesis holds for all lower-rank cases, the operator $\mathcal{W}_r$ calculates the generic Whittaker coefficient for the induced representation $I(\Delta(\tau_1,r_1)|\cdot|^{\frac{n_1-r_1}{2}+s_1}\boxtimes \Delta(\tau_2,r_2)|\cdot|^{\frac{n_2-r_2}{2}+s_2})$. By the induction hypothesis, this Whittaker coefficient is nonzero only if $r_1 \leq 1$ and $r_2\leq 1$. When corresponded to the order $r = r_1a_1 + r_2 a_2$ of the derivative, the only possibilities for $r$ are $r= a_1, a_2$ and $a_1+a_2$.
\end{proof}


\subsubsection{Case $\tau_1\cong\tau_2\cong\tau$}
In this case, by \cite[Theorem A]{zhang2022some}, the residues of the Eisenstein series $E_{\underline{s}}(\tilde{\varphi})$ at the simple poles $s_1-s_2 = \frac{n_1+n_2}{2}-\alpha$ for $\alpha\in \{0,1,\ldots,n_1-1\}$ is the irreducible quotient of the principal series $I(\tau_1,\tau_2;\underline{s})$. In this case, the intertwining operator $\mathcal{A}_w$ in (\ref{geomlemma})
may introduce extra zeros. We will analyze the possible zeros or poles introduced by the intertwining operator $\mathcal{A}_w$. 

\begin{proposition}\label{propres}
    The derivative $D^{(m)}$ of the residue of the Eisenstein series $E_{\underline{s}}$ at $s_1-s_2 = \frac{n_1+n_2}{2}-\alpha$, with $\alpha \in\{0,1,\ldots,\min\{n_1,n_2\}-1\}$ is possibly nonzero only if
    \begin{enumerate}
        \item $m = a$ or $2a$ when $\alpha > 0$,
        \item $m = a$ when $\alpha = 0$.
    \end{enumerate}
    However, if $n_2 = 1$, in which case the only allowed $\alpha$ is 0, then $D^{(a)}$ also vanishes.
\end{proposition}
\begin{proof}
    Since the integral operators $c_{P'\cap w^{-1}M w}, \mathcal{A}_w$ and $\mathcal{W}_s$ in Corollary \ref{advgeomlemma} are either over a compact domain or can be analytically continued to the whole complex plane, they all commute with the residue operator. The result from Proposition \ref{notequal} still holds, but extra vanishing may be introduced by the operator $\mathcal{A}_w$ and the residue operator. According to the proof of the Proposition \ref{notequal}, the nonzero derivatives $D^{(m)}$ correspond to the cases
    \begin{enumerate}
        \item If $m = a$, then $(r_1,r_2) = (1,0)$ or $(0,1)$.
        \item If $m = 2a$, then $(r_1,r_2) = (1,1)$.
    \end{enumerate}
    When $m = a$, for the term $I_w$ corresponding to the case $(r_1,r_2) = (1,0)$, the constant term operator sends the vector of the induced representation to the induced representation following the intertwining operator:
    \begin{align*}
        &I\left(\Delta(\tau_1,n_1-1)|\cdot|^{-\frac{1}{2}+s_1}\boxtimes \tau|\cdot|^{\frac{n_1-1}{2}+s_1}\boxtimes\Delta(\tau_2,n_2)|\cdot|^{s_2}\right)\\
        \xrightarrow{\mathcal{A}_w}&\\
        &I\left(\Delta(\tau_1,n_1-1)|\cdot|^{-\frac{1}{2}+s_1}\boxtimes\Delta(\tau_2,n_2)|\cdot|^{s_2}\boxtimes \tau|\cdot|^{\frac{n_1-1}{2}+s_1}\right).
    \end{align*}
    If $\alpha < n_1 -1$, the pole of the original Eisenstein series $E_{\underline{s}}$ at $s_1-s_2 = \frac{n_1+n_2}{2}-\alpha$ is still a pole of the Eisenstein series $E^{[n_1-1,n_2]}_{\underline{s}}$ constructed from the principal series $I\left(\Delta(\tau_1,n_1-1)|\cdot|^{-\frac{1}{2}+s_1}\boxtimes\Delta(\tau_2,n_2)|\cdot|^{s_2}\right)$ when restricted to the subgroup $GL_{n_1+n_2 - 1}$ on the top-left corner. The intertwining operator $\mathcal{A}_w$ has no pole unless when $\alpha = n_1-1$. When $\alpha = n_1-1$, $s_1-s_2 = \frac{n_1+n_2}{2}-\alpha$ is no longer a pole of $E_{\underline{s}}^{[n_1-1,n_2]}$. The intertwining operator can be written as the product
    $\mathcal{A}_w = r(w,\underline{s})N(w,\underline{s})$ (cf. \cite{moeglin1989spectre}) of the normalized intertwining operator $N(w,\underline{s})$ 
    with the normalization factor
    \[
        r(w,\underline{s}) = \prod_{\substack{i<j\\w(i)>w(j)}}\frac{L(\nu_i-\nu_j,\tau\times\hat\tau)}{L(1+\nu_i-\nu_j,\tau\times\hat\tau)\epsilon(\nu_i-\nu_j,\tau\times\hat\tau,\psi)}.
    \]
    In this case, assuming $s_1-s_2 = \frac{n_1+n_2}{2}-\alpha+t$, and after cancellation between numerators and denominators, the normalization factor
    \[
        r(w,\underline{s}) = \frac{L\left(\frac{n_1-n_2}{2}+s_1-s_2\right)}{L\left(\frac{n_1+n_2}{2}+s_1-s_2\right)}\times\epsilon\text{ factors} = \frac{L(n_1-\alpha+t)}{L(n_1+n_2-\alpha+t)}\times\epsilon\text{ factors}.
    \]
    When $\alpha = n_1-1$, the intertwining operator $\mathcal{A}_w$ has a simple pole at $s = \frac{n_1+n_2}{2}-\alpha$. \\ 
    
    In the case $(r_1,r_2) = (0,1)$, the intertwining operator $\mathcal{A}_w$ is a constant operator on any vector in the principal series $I\left(\Delta(\tau_1,n_1)|\cdot|^{s_1}\boxtimes \Delta(\tau_2,n_2-1)|\cdot|^{-\frac{1}{2}+s_2}\boxtimes \tau|\cdot|^{\frac{n_2-1}{2}+s_2}\right)$, and the whole term corresponding to this case in Corollary \ref{advgeomlemma} is killed by the residue operator if $\alpha = 0$.\\

    For $m = 2a$, the intertwining operator $\mathcal{A}_w$ intertwines between the following two principal series:
    \begin{align*}
        &I\left(\Delta(\tau_1,n_1-1)|\cdot|^{-\frac{1}{2}+s_1}\boxtimes \tau|\cdot|^{\frac{n_1-1}{2}+s_1}\boxtimes\Delta(\tau_2,n_2-1)|\cdot|^{-\frac{1}{2}+s_2}\boxtimes \tau|\cdot|^{\frac{n_2-1}{2}+s_2}\right)\\
        \xrightarrow{\mathcal{A}_w}&\\
        &I\left(\Delta(\tau_1,n_1-1)|\cdot|^{-\frac{1}{2}+s_1}\boxtimes\Delta(\tau_2,n_2-1)|\cdot|^{-\frac{1}{2}+s_2}\boxtimes \tau|\cdot|^{\frac{n_1-1}{2}+s_1}\boxtimes \tau|\cdot|^{\frac{n_2-1}{2}+s_2}\right).
    \end{align*}
    In this case, assuming $s_1-s_2 = \frac{n_1+n_2}{2}-\alpha+t$, and after cancellation between numerators and denominators, the normalization factor
    \[
        r(w,\underline{s}) = \frac{L\left(\frac{n_1-n_2}{2}+1+s_1-s_2\right)}{L\left(\frac{n_1+n_2}{2}+s_1-s_2\right)}\times\epsilon\text{ factors} = \frac{L(n_1+1-\alpha+t)}{L(n_1+n_2-\alpha+t)}\times\epsilon\text{ factors}.
    \]
    This normalization factor will not introduce any pole to the intertwining operator. The only complication occurs when $\alpha = 0$, in which case $s_1-s_2 = \frac{n_1+n_2}{2}$ is no longer a pole of the resulting Eisenstein series, thus the derivative $D^{(2a)}$ will be killed by the residue operator.
\end{proof}

\subsection{Whittaker Support of Eisenstein Series with Discrete Spectrum Data}
For the principal series $I(\tau_1,\tau_2;\underline{s})$ induced from discrete spectrum data, if $\tau_1\not\cong\tau_2$, by Proposition \ref{notequal}, the allowed derivatives are $D^{(a_1)}$, $D^{(a_2)}$ and $D^{(a_1+a_2)}$, sending the Eisenstein series on $GL_n$ to an Eisenstein series $E^{[n_1-1,n_2]}$, $E^{[n_1,n_2-1]}$ and $E^{[n_1-1,n_2-1]}$, respectively. Applying Proposition \ref{notequal} iteratively, a degenerate Whittaker coefficient $\mathcal{W}_{H,\phi_\lambda}$ is zero if the partition $\lambda$ contains numbers other than $a_1,a_2$ and $a_1+a_2$.
\begin{proposition}\label{isobaricnonequal}
    Without loss of generality, assuming $n_1\geq n_2$, for $\tau_1\not\cong \tau_2$ two cuspidal automorphic representations of $GL_{a_1}, GL_{a_2}$, respectively, the Eisenstein series $E_{\underline{s}}$ associated to the principal series $I(\tau_1,\tau_2;\underline{s})$ has no poles when $\mathrm{Re}(s_1-s_2)> 0$, and the maximal orbit that allows a degenerate Whittaker coefficient corresponds to the partition $((a_1+a_2)^{n_2}a_1^{n_1-n_2})$.
\end{proposition}
\begin{proof}
    If $(n_1,n_2) = (1,0)$ or $(0,1)$, the result reduces to the case of cuspidal automorphic representations. If $(n_1,n_2) = (1,1)$, the maximal orbit allowing a Whittaker coefficient is the generic orbit of $GL_{a_1+a_2}$, and hence corresponds to the partition $(a_1+a_2)$. If $(n_1,n_2) = (2,0)$ or $(0,2)$, the result reduces to the cases of generalized Speh representations, and the maximal orbit allowing a Whittaker coefficient corresponds to the partition $(a_1^2)$ and $(a_2^2)$, respectively. We now apply induction on $n = n_1+n_2$ to prove the general case. Among all the partitions containing $a_1,a_2$ and $a_1+a_2$, the largest partitions are those with prefix $(a_1+a_2)^t$ for some $t>0$. Assuming $n_1 > n_2$, for any derivative $D^{(a_1)}$, by Proposition \ref{geomlemma}, 
    \[
        D^{(a_1)}E_{\underline{s}}(g) = E^{[n_1-1,n_2]}\boxtimes \mathcal{W}_{a_1} + \begin{cases}
            0 & \text{if }a_1\not = a_2\\
            E^{[n_1,n_2-1]}\boxtimes \mathcal{W}_{a_2}& \text{if }a_1 = a_2
        \end{cases}
    \]
    By the induction hypothesis, each one of the two terms on the right hand side has Whittaker support corresponding to the partition $[(a_1+a_2)^{n_2}a_1^{n_1-n_2-1}]$ and $[(a_1+a_2)^{n_2-1}a_1^{n_1-n_2+1}]$, respectively. In either case, the second term is killed by the degenerate Whittaker coefficient operator corresponding to the partition $[(a_1+a_2)^{n_2}a_1^{n_1-n_2-1}]$. If $n_1 = n_2$, applying the derivative operator $D^{(a_1+a_2)}$, we have
    \[
        D^{(a_1+a_2)}E(\underline{s},g) = E^{[n_1-1,n_2-1]}\boxtimes \mathcal{W}_{a_1+a_2}
    \]
    with $\mathcal{W}_{a_1+a_2}$ the Whittaker coefficient of the cuspidal Eisenstein series $E^{[1,1]}$. By the induction hypothesis, the maximal orbit allowing a degenerate Whittaker coefficent for $E^{[n_1-1,n_2-1]}$ is $[(a_1+a_2)^{n_1-1}]$. Since the only allowed derivatives are $D^{(a_1)}, D^{(a_2)}$ and $D^{(a_1+a_2)}$, the maximal orbit allowing a degenerate Whittaker coefficient is thus $[(a_1+a_2)^{n_1}]$.
\end{proof}
In the situation when $\tau_1 = \tau_2$, we can calculate the Whittaker support of the residue of the Eisenstein series at $s = \frac{n_1+n_2}{2} - \alpha$. 
\begin{proposition}
    The Whittaker support of the residue of the Eisenstein series $\mathrm{Res}_{s = \frac{n_1+n_2}{2} - \alpha}E_{\underline{s}}(g)$ is $[a^{n-\alpha}]+[a^{\alpha}]$.
\end{proposition}
\begin{proof}
    The first residue appears in the the case when $(n_1,n_2) = (1,1)$. By Proposition \ref{propres}, the only allowed $\alpha$ with a residue at $s = \frac{n_1+n_2}{2}-\alpha$ appears when $\alpha = 0$, from which we obtain a generalized Speh representation $\Delta(\tau,2)$. The only allowed derivative for the residue in this case is $D^{(a)}$. Through the proof of Proposition \ref{propres}, there is a pole of the intertwining operator $\mathcal{A}_w$ for the case $(r_1,r_2) = (1,0)$, and it is clear that the derivative of the residue $D^{(a)}\mathrm{Res}_{s=\frac{n_1+n_2}{2}-\alpha}E_{\underline{s}}$ is a vector in $\tau$ after restricted to $GL_{a}$. Assuming the proposition holds for all smaller $n = n_1+n_2$, the derivative $D^{(a)}E_{\underline{s}}$ can be written as a sum of two terms:
    \[
        D^{(a)}E_{\underline{s}} = E_{[n_1-1,n_2]}\boxtimes \mathcal{W}_{a} + E_{[n_1,n_2-1]}\boxtimes \mathcal{W}_{a}'.
    \]
    If $\alpha < n_1-1$, the pole $s = \frac{n_1+n_2}{2}-\alpha$ is still a pole of the new Eisenstein series $E_{[n_1-1,n_2]}$, but if $\alpha = n_1 - 1$, it is no longer a pole of $E_{[n_1-1,n_2]}$ but instead a pole of the intertwininig operator $\mathcal{A}_w$. If $\alpha >0$ the pole $s = \frac{n_1+n_2}{2}-\alpha$ is still a pole of $E_{[n_1,n_2-1]}$. Therefore, when $n_2 > 1$, 
    \[
        D^{(a)}\mathrm{Res}_{s = \frac{n_1+n_2}{2}-\alpha}E(\underline{s},g) = \begin{cases} 
            \mathrm{Res}_{s =  \frac{n_1+n_2}{2}-\alpha}E_{[n_1-1,n_2]}\boxtimes \mathcal{W}_2  & \alpha = 0\\
            \mathrm{Res}_{s = \frac{n_1+n_2}{2}-\alpha}E_{[n_1-1,n_2]}\boxtimes \mathcal{W}_1 + \mathrm{Res}_{s = \frac{n_1+n_2}{2}-\alpha}E_{[n_1,n_2-1]}\boxtimes \mathcal{W}_2 & 0<\alpha < n_1-1\\
            c\mathrm{Res}_{s = \frac{n_1+n_2}{2}-\alpha}E_{[n_1-1,n_2]}\boxtimes \mathcal{W}_1 + \mathrm{Res}_{s = \frac{n_1+n_2}{2}-\alpha}E_{[n_1,n_2-1]}\boxtimes \mathcal{W}_2 & \alpha = n_1-1
        \end{cases}
    \]
    By the induction hypothesis, for $n_1 > 1$, if $\alpha = n_1 -1$, the Whittaker supports of each of the two terms in $D^{(a)}\mathrm{Res}_{s = \frac{n_1+n_2}{2}-\alpha}E(\underline{s},g)$ are $[a^{n_1-1}]+[a^{n_2}]$ and $[a^{n-\alpha}]+[a^{\alpha-1}] = [a^{n_1-2}]+[a^{n_2+1}]$, respectively. If $\alpha = 0$, the Whittaker support of $D^{(a)}\mathrm{Res}_{s = \frac{n_1+n_2}{2}-\alpha}E(\underline{s},g)$ is $[a^{n-1}]$. In other cases, each term admitting Whittaker coefficient corresponds to $[a^{n-\alpha}]+[a^{\alpha-1}]$ and $[a^{n-\alpha-1}]+[a^{\alpha}]$, respectively. Since $0\leq \alpha \leq \min\{n_1,n_2\}-1$, it turns out that any case would not allow a Whittaker support corresponding to an orbit larger than $[a^{n-\alpha}] +[a^\alpha]$.
\end{proof}

\subsection{Eulerianity of Whittaker Coefficients}
For any Whittaker pair $(H,\phi)$ and any pure tensor $f = \bigotimes_v f_v $ in an automorphic representation $\pi = \bigotimes_v\pi_v$, we say its Whittaker coefficient $\mathcal{W}_{H,\phi}(f)$ is \emph{Eulerian} if $\mathcal{W}_{H,\phi}(f)$ factorizes into a tensor product
\[
    \mathcal{W}_{H,\phi}(f)\left(\prod_v g_v\right) = \bigotimes_v \mathcal{W}_v(f)\left(g_v\right)
\]
with $\mathcal{W}_v$ a function on the group $G(k_v)$. By multiplicity one theorems in \cite{shalika1974multiplicity},\cite{piatetski1979multiplicity}, when $\phi$ is a generic character of the maximal unipotent subgroup $U$, the Whittaker coefficient is Eulerian. A larger variety of cases are considered in \cite{gourevitch2021eulerianity}. In particular, in \cite[Theorem F]{gourevitch2021eulerianity}, the case when $\pi$ is a discrete spectrum automorphic representation is considered. We will be applying our method to cover this case as well as Eisenstein series induced from discrete spectrum data.

\begin{theorem}
    For $\pi$ in the discrete spectrum of $GL_n(\mathbb{A})$, any degenerate Whittaker coefficient of $\pi$ is Eulerian.
\end{theorem}
\begin{proof}
    By (\ref{geomlemma}), the only nonzero degenerate Whittaker coefficient of $f\in\Delta(\tau,n)$ is associated to the orbit $[a^n]$. This degenerate Whittaker coefficient is Eulerian because $D^{(a)}f$ decomposes as the product of a vector $f'\in \Delta(\tau,n-1)$ and a top Fourier coefficient of cuspidal representation $\tau$. The Eulerianity result follows from applying $D^{(a)}$ iteratively and the Eulerianity of the top Whittaker coefficient of a cuspidal automorphic representation.
\end{proof}
\begin{theorem}
    For Eisenstein series constructed from the isobaric sum representation $I(\tau_1,\tau_2;s)$ with $\tau_1\not\cong \tau_2$, if $a_1\neq a_2$, then any nonzero Whittaker coefficent is Eulerian.
\end{theorem}
\begin{proof}
    The only nonzero Whittaker coefficients of the Eisenstein series are those which contains only $a_1, a_2$ and $a_1+a_2$. In the proof of Proposition \ref{isobaricnonequal}, if $a_1\neq a_2$, the derivative $D^{(a_1)},D^{(a_2)}$ of the Eisenstein series contains only one term $E^{[n_1-1,n_2]}\otimes \mathcal{W}_{a_1}$ or $E^{[n_1,n_2-1]}\otimes \mathcal{W}_{a_2}$, with $\mathcal{W}_{a_1},\mathcal{W}_{a_2}$ top Whittaker coefficients of $\tau_1,\tau_2$, respectively, and thus are Eulerian functions on $GL_{a_1}, $ and $GL_{a_2}$, respectively. Similarly, for the derivative $D^{(a_1+a_2)}$, the derivative of the Eisenstein series contains only one term $E^{[n_1-1,n_2-1]}\boxtimes\mathcal{W}_{a_1+a_2}$, where $\mathcal{W}_{a_1+a_2}$ is a top Whittaker coefficient for a cuspidal Eisenstein series on $GL_{a_1+a_2}$. The Euleranity of any nonzero Whittaker coefficient of the Eisenstein series follows from applying this process iteratively.
\end{proof}

\begin{theorem}
    In all cases discussed in this paper, the top degenerate Whittaker coefficient is Eulerian.
\end{theorem}
\begin{proof}
    The partition associated to the largest orbit allowing a top degenerate Whittaker coefficient is $[(a_1+a_2)^{\min\{n_1,n_2\}} a_i^{|n_1-n_2|}]$. The derivative $D^{(a_1+a_2)}E_{\underline{s}}$  
    is an Eulerian function $E^{[n_1-1,n_2-1]}\boxtimes\mathcal{W}_{a_1+a_2}$. After applying $D^{(a_1+a_2)}$ multiple times such that the resulting function belongs to the generalized Speh representation $\Delta(\tau_i,|n_1-n_2|)$ , whose Whittaker support on $GL_{a_i|n_1-n_2|}$ is $[a_i^{|n_1-n_2|}]$. By \cite[Theorem F]{gourevitch2021eulerianity}, it is an Eulerian function on $GL_{a_i|n_1-n_2|}$.
\end{proof}

\bibliographystyle{alpha}
\bibliography{main}
\end{document}